\documentclass[11pt]{amsart}

\usepackage[dvipsnames]{xcolor}
\usepackage{enumitem}
\usepackage{amssymb,graphicx,mathtools}
\usepackage{float}
\usepackage{tikz}
\usetikzlibrary{calc,arrows.meta,decorations.markings}
\usepackage[margin=1in]{geometry}
\usepackage{etoolbox}

\usepackage[square,longnamesfirst]{natbib}

\theoremstyle{plain}
\newtheorem{theorem}{Theorem}[section]
\newtheorem{proposition}[theorem]{Proposition}

\newtheorem{lemma}[theorem]{Lemma}
\theoremstyle{definition}
\newtheorem{remark}[theorem]{Remark}
\newtheorem{problem}{Problem}

\usepackage[
citecolor=PineGreen,
colorlinks=true,
linkcolor=RoyalBlue,
urlcolor=BrickRed
]{hyperref}

\graphicspath{{figures/}}
\allowdisplaybreaks

\newcommand{\Z}{\mathbb Z}
\newcommand{\R}{\mathbb R}
\newcommand{\T}{\mathbb T}
\newcommand{\E}{\mathbb E}
\renewcommand{\P}{\mathbb P}
\newcommand{\1}{\mathbf 1}
\newcommand{\ip}[2]{\langle #1,#2\rangle}
\newcommand{\norm}[1]{\lVert #1\rVert}
\newcommand{\abs}[1]{\lvert #1\rvert}
\newcommand{\dd}{\,\mathrm d}
\newcommand{\Hs}{\mathcal H}

\newcommand{\Hminus}[1]{\norm{#1}_{-1,\lambda}}
\newcommand{\Hplus}[1]{\norm{#1}_{1,\lambda}}
\newcommand{\defn}[1]{\textbf{#1}}

\tikzset{
	ledge/.style={draw=black!22,line width=0.5pt},
	lvert/.style={circle,fill=black,inner sep=1.1pt}
}

\makeatletter
\patchcmd{\@settitle}{\uppercasenonmath\@title}{\Large}{}{}
\patchcmd{\@setauthors}{\MakeUppercase}{\large}{}{}
\makeatother

\title{The randomly oriented Manhattan lattice in 2D is transient}
\author{Ahmed Bou-Rabee}
\author{Yuval Peres}

\begin{document}
	
	\begin{abstract}
		Independently orient each horizontal and vertical line of $\Z^2$ by a fair
		coin.  A walker chooses one of the two lines through its current position
		with equal probability and takes one step in the direction of that line.
		Redner \citeyearpar{Red89} introduced this walk as a model of transport in
		an isotropic random velocity field and predicted that its root-mean-square
		displacement grows like $n^{2/3}$.  We prove that the walk is transient
		almost surely.
		The proof is a short variational argument.
	\end{abstract}
	
	\maketitle
	
	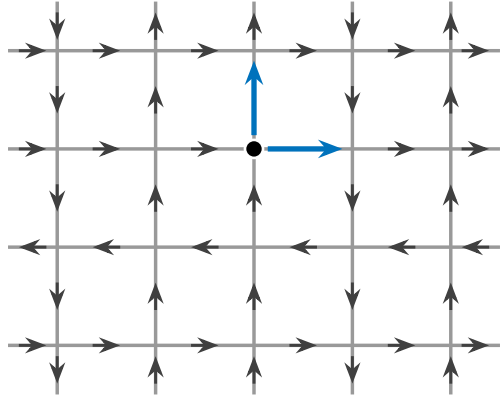
\begin{figure}[!b]
		\centering
		\begin{tikzpicture}[x=1.3cm,y=1.3cm,
			street/.style={draw=black!40,line width=1.3pt},
			head/.style={draw=black!75,line width=1.1pt,-{Stealth[length=2.8mm]}},
			move/.style={draw=RoyalBlue,line width=2.0pt,-{Stealth[length=3.2mm]}},
			site/.style={circle,fill=black,draw=white,line width=1.0pt,inner sep=2.4pt}]
			\foreach \y in {0,...,3} \draw[street] (-0.5,\y) -- (4.5,\y);
			\foreach \x in {0,...,4} \draw[street] (\x,-0.5) -- (\x,3.5);
			\foreach \y/\ms in {0/{-0.25,0.5,1.5,2.5,3.5,4.25},
				2/{-0.25,0.5,1.5,3.5,4.25},
				3/{-0.25,0.5,1.5,2.5,3.5,4.25}}
			\foreach \m in \ms \draw[head] (\m-0.14,\y) -- (\m+0.14,\y);
			\foreach \m in {-0.25,0.5,1.5,2.5,3.5,4.25}
			\draw[head] (\m+0.14,1) -- (\m-0.14,1);
			\foreach \x/\ms in {1/{-0.25,0.5,1.5,2.5,3.25},
				2/{-0.25,0.5,1.5,3.25},
				4/{-0.25,0.5,1.5,2.5,3.25}}
			\foreach \m in \ms \draw[head] (\x,\m-0.14) -- (\x,\m+0.14);
			\foreach \x in {0,3} \foreach \m in {-0.25,0.5,1.5,2.5,3.25}
			\draw[head] (\x,\m+0.14) -- (\x,\m-0.14);
			\draw[move] (2.14,2)--(2.90,2);
			\draw[move] (2,2.14)--(2,2.90);
			\node[site] at (2,2) {};
		\end{tikzpicture}
		\caption{One fair coin directs each line, so every edge of a line points the
			same way.  The two blue arrows are the possible steps from the marked site.}
		\label{fig:model}
	\end{figure}
	
	\section{Introduction}\label{sec:introduction}
	
	The \defn{randomly oriented Manhattan lattice} is the square lattice with
	every horizontal and vertical line given a direction by an independent fair
	coin.  We study the random walk that at each step chooses one of the two
	lines through its position with equal probability and moves one unit in that
	line's direction (Figure~\ref{fig:model}). Once a line is oriented, every later step along it has the same direction.
	Figure~\ref{fig:path} compares the walk with simple random walk.
	
	\begin{figure}[t]
		\centering
		
		\includegraphics[width=0.5\textwidth]{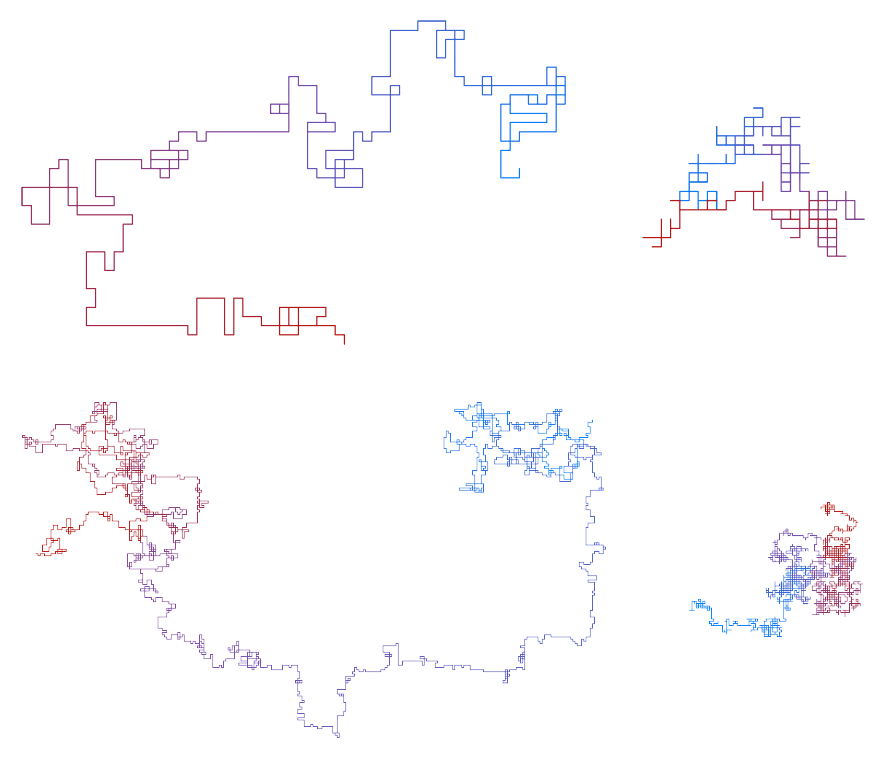}
		\caption{The randomly oriented Manhattan walk on the left and simple random
			walk on the right, after $400$ steps in the top row and $4000$ steps in the
			bottom row, colored by time from blue to red.  The two panels of each row are
			drawn to the same scale.}
		\label{fig:path}
	\end{figure}
	
	\citet*{Red89} introduced this walk and predicted a
	non-Gaussian displacement profile on the scale $n^{2/3}$.  In particular,
	this prediction implies transience of the walk. 
	
	\subsection{Main results}\label{ssec:main-results}
	
	Write $e_1,e_2$ for the standard basis of $\Z^2$, and record the
	line orientations by
	$\omega\colon\Z^2\times\{1,2\}\to\{-1,1\}$: the line through $z$ parallel
	to $e_i$ points along $\omega(z,i)e_i$.  We call $\omega$ the
	\defn{environment} and write $\P$ for its law.
	
	From $z$ the walker steps to
	\begin{equation}\label{eq:rule}
		z+\omega(z,1)e_1\qquad\text{or}\qquad z+\omega(z,2)e_2,
		\qquad\text{with probability $1/2$ each.}
	\end{equation}
	This can be understood as a random walk with a random drift which is divergence free.

	Write $P^\omega_z$ for the quenched law of the walk started at $z$ in a fixed
	environment $\omega$, and write $p^\omega_n(z,z')$ for the probability under
	$P^\omega_z$ that the walk is at $z'$ after $n$ steps.
	Averaging over $\omega$ gives the annealed transition probability
	$\overline p_n(z,z')=\E p^\omega_n(z,z')$.
	
	\begin{theorem}\label{thm:main}
		For $\P$-almost every environment $\omega$,
		\[
		\sum_{n\geq0}p^\omega_n(x,x)<\infty\qquad\text{for every }x\in\Z^2\, .
		\]
	\end{theorem}
	
	We prove the above theorem by analyzing a continuous time version of the
	walk, which jumps across each of its two outgoing edges at rate $1$. 
	Write $p^\omega_t(z,z')$ for its transition probabilities and set
	$\overline p_t(0,0)=\E p^\omega_t(0,0)$.  We deduce Theorem~\ref{thm:main}
	from the following annealed bound.
	
	\begin{theorem}\label{thm:annealed}
		For every $\lambda\in(0,1]$,
		\[
		\int_0^\infty e^{-\lambda t}\overline p_t(0,0)\dd t\leq2048\, ,
		\]
		and hence $\int_0^\infty\overline p_t(0,0)\dd t<\infty$.
	\end{theorem}
	
	Note that transience does not hold for every deterministic environment.
	Indeed, take the alternating orientations, $\omega(z,1)=(-1)^{z_2}$ and
	$\omega(z,2)=(-1)^{z_1}$.  This walk can be mapped to simple random walk on
	$\Z^2$ as detailed in \citet*{BH23}, who report
	this as an observation of Guillotin-Plantard.
	
	\subsection{Notation and conventions}\label{ssec:notation}
	
	\begin{itemize}
		\item $\P$ is the law of the environment and $\E$ is the corresponding
		expectation.
		\item $\ip\cdot\cdot$ and $\norm\cdot$ denote the inner product and norm of the
		$L^2$ space which is indicated. 
		\item $c>0$ and $C>0$ are finite constants which may change from line to
		line.  Numerical constants in the quantitative bounds are displayed
		explicitly.
	\end{itemize}
	
	\subsection{Proof overview}\label{ssec:proof-overview}
	
	For reversible walks, the usual technique to show transience is to exhibit a
	unit flow from a vertex to infinity with finite energy, see, for example,
	\citep*[Theorem~2.11]{LP16}.  In our nonreversible setting, we instead use a
	variational formula for the resolvent, given, for example, in
	\citep*[Theorem~4.1]{KLO12}.  Variational principles for operators that are
	not self-adjoint are standard in homogenization; see
	\citet*[Chapter~10]{AKM19}, \citet*[(2.4)]{ABK24}, and
	\citet*[(2.32)]{ABK26}.  	
	
	The variational formula bounds the annealed Green function using a test
	function at each spatial frequency.  We build the test function from the
	line orientations.  It is not a flow and satisfies no additional constraint:
	every square-integrable function of the orientations can be used.
	The difficulty is to choose one that makes the bound small.
	
	The zero test function in this variational principle gives the same bound as for simple random walk, which is too weak: its
	frequency integral diverges logarithmically. To improve on this, we must exploit the enhancement of the drift by careful choice of test function.  
	This choice is made from a family of functions formed from products of one and three line signs, indexed by one even function of a single
	frequency. As we will see, a careful optimization of the even function will make the frequency integral converge.

	\medskip
	\noindent\textbf{Step 1: A variational bound at each frequency.}
	Fix $\lambda\in(0,1]$.  The law of the environment is translation invariant,
	so we may Fourier transform in the position of the walk and treat each
	frequency on its own.  Write $\T\coloneqq(-\pi,\pi]$ with normalized Haar measure
	$\dd m$, and $\dd m(p)=\dd m(p_1)\dd m(p_2)$ on $\T^2$.  Write $T_x$ for the shift of the
	environment by $x$.  At a frequency $p\in\T^2$ the generator of the walk acts on
	$L^2(\P)$ as $S_p+A_p$, where
	\begin{align*}
		S_p&\coloneqq\frac12\sum_{i=1}^2
		\bigl(e^{ip_i}T_{e_i}+e^{-ip_i}T_{-e_i}-2I\bigr),\\
		A_p&\coloneqq\frac12\sum_{i=1}^2\omega(0,i)
		\bigl(e^{ip_i}T_{e_i}-e^{-ip_i}T_{-e_i}\bigr)
	\end{align*}
	(Proposition~\ref{prop:generator}).  Here $S_p$ is self-adjoint and
	nonpositive and $A_p$ is skew-adjoint.  Let $H\coloneqq\lambda I-S_p$ and
	$r_\lambda(p)\coloneqq\operatorname{Re}\ip1{(H-A_p)^{-1}1}$. The annealed Green function is the
	average of $r_\lambda$ over the frequencies,
	\[
	\int_0^\infty e^{-\lambda t}\overline p_t(0,0)\dd t
	=\int_{\T^2}r_\lambda(p)\dd m(p)
	\]
	(Proposition~\ref{prop:generator}).  Every $g\in L^2(\P)$ gives an upper
	bound (Lemma~\ref{lem:variational-bound}),
	\begin{equation}\label{eq:overview-bound}
		r_\lambda(p)\leq
		\ip g{Hg}+\ip{1-A_pg}{H^{-1}(1-A_pg)}\, .
	\end{equation}
	The first term is the energy of $g$, and the second is
	$\norm{H^{-1/2}(1-A_pg)}^2$.  We seek $g$ for which both terms are small.
	
	Taking $g=0$ makes the energy zero and leaves $1-A_pg=1$, which
	gives the simple random walk bound $r_\lambda(p)\leq[\lambda+\theta(p)]^{-1}$
	with $\theta(p)\coloneqq2-\cos p_1-\cos p_2$.  Averaging this upper bound
	over $p$ gives order $\log(1/\lambda)$.  We therefore need a nonzero test function to
	obtain a bound whose integral stays finite as $\lambda\downarrow0$. 
	
	\medskip
	\noindent\textbf{Step 2: A family of test functions.}
	To improve the bound from $g=0$, we take $g=f+k$, where
	\begin{equation}\label{eq:overview-test-family}
	  \begin{split}
	    f(\omega)&=\sum_{y\in\Z}f_y\omega(ye_2,1),\\
	    k(\omega)&=\sum_{\substack{x,y,y'\in\Z\\y<y'}}
	      k_{y,y',x}\omega(ye_2,1)\omega(y'e_2,1)\omega(xe_1,2)\, .
	  \end{split}
	\end{equation}
	Here $\omega(z,i)$ is the orientation of the line through $z$ parallel
	to $e_i$.  Thus $f$ is linear in the horizontal orientations, while $k$
	combines products of two distinct horizontal orientations and one vertical
	orientation.  We take the coefficient arrays to be square summable, so
	that $f,k\in L^2(\P)$.
	
	The role of $f$ is to make the constant component of $A_pg$
	close to $1$.  Recall that the Fourier--Walsh expansion writes each
	$F\in L^2(\P)$ in the
	orthonormal basis of products of orientations $\omega(z,i)$ of finitely
	many distinct lines; see, for example, \citet*[Theorem~1.5]{ODo14}.
	The formula for $A_p$ in Step~1 shifts these products and multiplies
	them by an orientation.  Multiplication either introduces that
	orientation or cancels it if already present, since $\omega(z,i)^2=1$.
	Thus only terms containing a single orientation can produce a
	constant term under $A_p$.  Note that \citet*[(2.1)]{LTV18} also use linear functions of the horizontal
	orientations to bound the Laplace transform of mean-square displacement
	from below.
	
	Each term of $f$ contains one orientation, and $A_p$ either removes
	that orientation or introduces a second one.  Thus $A_pf$ has only
	constant and quadratic components.  It remains to control the quadratic
	terms in $1-A_pf$.  We bound the
	contribution from products of two horizontal orientations directly.
	The role of $k$ is to reduce the contribution from products of a horizontal and a vertical
	orientation. We index the coefficients $(f_y)_{y\in\Z}$
	in Fourier coordinates using a real even function $\varphi\in L^2(\T,m)$.
	The correction $k$ also contributes energy and creates quartic terms
	in $A_pk$.  Lemma~\ref{lem:parity} constructs $k$ from $\varphi$ and
	shows that the bound for products of a horizontal and a vertical
	orientation gains a logarithmic term in its denominator, even after
	these additional contributions are included.  Thus $\varphi$ specifies $g=f+k$. The rest of the proof is to choose $\varphi$. 
	
	\medskip
	\noindent\textbf{Step 3: Optimizing and integrating over frequencies.}
	We choose $\varphi$ to balance cancellation of the constant term
	against the energy and the nonconstant terms in
	\eqref{eq:overview-bound}.  Lemma~\ref{lem:effective-energy} bounds
	$r_\lambda(p)$ by a quadratic functional of $\varphi$.  We choose
	$\varphi$ to minimize this functional; Section~\ref{sec:fixed-frequency}
	carries out the minimization. In particular, Proposition~\ref{prop:frequency} gives, for
	$a\coloneqq\max\{\abs{p_1},\abs{p_2}\}$,
	\begin{equation*}
	  r_\lambda(p)\leq\frac{4096}{a^2[1+\log(1/a)]^{3/2}}
	  \qquad(\sqrt\lambda<a<1/4)\, .
	\end{equation*}
	Since $3/2>1$, the integral of this bound over
	$\sqrt\lambda<a<1/4$ stays bounded as $\lambda\downarrow0$.
	The estimate from $g=0$ gives a uniformly bounded contribution from
	the remaining frequencies.  Letting $\lambda\downarrow0$ therefore
	gives a finite annealed Green function, proving
	Theorem~\ref{thm:annealed}.
	
	\subsection{Related results}\label{ssec:related-work}
	
	\subsubsection{Randomly oriented Manhattan lattices}\label{sssec:oriented}
	As discussed above, the displacement scale $n^{2/3}$ predicted by \citet*{Red89}
	would imply transience of the walk if accompanied by an annealed local limit: the return
	probability would then have order $n^{-4/3}$.  Simulations support this decay
	\citep*[Section~3.2]{MNOV20}; further discussions of the displacement
	prediction appear in \citet*{BG90} and
	\citet*[Section~3.5]{Pen20}.  The transience conjecture is attributed
	to Guillotin-Plantard in \citet*[Appendix~A]{Bos19}.
	
	In the mathematics literature, the Laplace transform of the annealed
	mean-square displacement has been bounded rigorously by 
	\citet*[Theorem~1.1]{LTV18} in dimensions two and three.  In dimension three,
	\citet*[Theorem~1.1]{Ngu26} obtained order
	$\lambda^{-2}\sqrt{\abs{\log\lambda}}$, up to multiplicative factors
	$(\log\abs{\log\lambda})^{\pm(2+\varepsilon)}$ for every
	$\varepsilon>0$ and all sufficiently small $\lambda$.
	These estimates establish superdiffusivity in the Laplace-transform sense, but do not immediately yield bounds on mean-square displacement. Such bounds alone do not control the probability of returning to the origin.
	
	When the horizontal orientations are independent and uniform on
	$\{-1,1\}$ and the vertical orientations alternate deterministically,
	the walk is almost surely transient \citep*[Theorem~9]{Bos19}.
	On the randomly oriented Manhattan lattice, following both orientations at every
	step gives the deterministic update
	$z\mapsto z+\omega(z,1)e_1+\omega(z,2)e_2$; this walk is eventually
	two-periodic for $\P$-almost every environment
	\citep*[Theorem~1]{CHT19}.  Quantum and classical network models on
	the alternating deterministic lattice are studied in
	\citet*{BOC03}.
	
	Our proof uses the nonreversible resolvent variational formula
	\citep*[Theorem~4.1]{KLO12} that \citet*[Section~2]{LTV18} use for displacement bounds.
	Polynomial test functions similar to ours also appear in the study of asymmetric exclusion
	\citep*[Sections~3--4]{LQSY04} \citep*[Section~2]{Yau04} and the
	stationary one-dimensional self-repelling Brownian polymer
	\citep*{TTV12}.  Variational formulas for
	resolvents are discussed in \citet*[(39)]{Tot18}, and their connection
	with capacities and transience in \citet*[Theorem~2.4]{GL14}.
	
	\subsubsection{Partially oriented lattices}
	
	If only the horizontal lines are oriented, the walker can move up,
	down, or in the direction of its current horizontal line.  Choosing
	uniformly among these three moves makes the vertical coordinate a
	lazy simple random walk independent of the orientations.  One can
	therefore describe the visits to horizontal lines before using their
	directions to determine the horizontal displacement.  This is a
	discrete version of the Matheron--de Marsily model
	\citep*{MdM80}.  For the walk studied here, each coordinate
	depends on the orientations encountered by the other, so the same
	separation is unavailable.
	
	With independent horizontal orientations uniform on $\{-1,1\}$,
	the partially oriented walk is almost surely transient
	\citep*[Theorem~1.8]{CP03} and has zero speed
	\citep*[Theorem~2.11]{CP04}.  The horizontal displacement grows on
	a different scale from the vertical displacement:
	\citet*[Theorem~2]{GLN08F} proved an annealed joint functional limit
	theorem with scales $n^{3/4}$ and $n^{1/2}$, respectively.
	The limiting coordinates are dependent, even though the vertical
	motion is independent of the orientations.  The annealed return
	probability has order $n^{-5/4}$ 
	\citep*[Theorem~20]{CGPS11}.
	
	For the partially oriented walk, \citet*[Theorem~1]{GLN08T} extend almost-sure transience from independent
	horizontal orientations uniform on $\{-1,1\}$ to orientations whose
	probabilities are sampled along a stationary orbit of a dynamical system.  The orientations
	are independent conditional on the orbit, but can be dependent after
	averaging over it.  Their theorem assumes that each orientation has
	mean zero and that the probabilities satisfy the integrability
	condition (2.3).  \citet*{Bre17} allows all transition
	probabilities to depend on the layer.  In such models, changing only
	departures from the horizontal axis can turn a transient walk into
	a recurrent one \citep*[Corollary~2.6]{BHP22}; see
	also \citet*[Proposition~7.8]{Bre17}.
	
	\subsubsection{Diffusion in incompressible random flows}
	\label{sssec:divfree}
	
	Our walk has divergence-free drift, making it a discrete analogue of a
	Brownian particle transported by a stationary incompressible random flow
	\citep*{BGKPR90}.  For a class of incompressible random drifts with covariance
	decay $R^{-2(1-\gamma)}$ at distance $R$, \citet*{ABK26} obtained moment
	estimates over the environment for the quenched mean-square displacement
	at scale $t^{2/(2-\gamma)}$, for sufficiently small positive $\gamma$.
	For the $d$-dimensional version of our walk, the covariance of each drift
	component is constant along its coordinate axis and zero elsewhere.
	Its average over a box of side $R$ centered at the origin therefore has
	order $R^{1-d}$.
	Matching these exponents gives the heuristic correspondence
	\begin{equation}\label{eq:gamma-dimension}
		\gamma=\frac{3-d}{2}\, .
	\end{equation}
	The renormalization-group predictions of \citet*{BG90}, as recalled in
	\citet*{ABK26}, then give mean-square displacement of order $t^{4/3}$ in
	dimension two, $t(\log t)^{1/2}$ in dimension three, and $t$ in higher
	dimensions.  In dimension two this recovers Redner's displacement exponent.
	It remains unclear how to adapt the methods of \citet*{ABK26} to our
	setting, since their argument requires sufficiently small positive $\gamma$.

	\subsection{Formalization}\label{ssec:formalization}
	
	Every theorem, lemma and proposition of this paper has been formalized
	in Lean~4 and proved from Mathlib, together with the recurrence assertion
	in Problem~\ref{prob:which}.  No proof uses \texttt{sorry} or any axiom
	beyond the three of Lean's own logic.  The development is available at
	\url{https://github.com/nitromannitol/Manhattan-Transience}.
	
	\section*{Acknowledgments}
	The research of Y. Peres was supported by National Natural Science Foundation of China grant
	RFIS-W2531011.  We acknowledge use of ChatGPT 5.6.
	
	\section{The variational bound}\label{sec:variational}
	
	\subsection{Fourier transform}\label{ssec:fourier}
	
	For a fixed environment and $h\colon\Z^2\to\R$, the generator is
	\begin{equation}\label{eq:discrete-drift}
		\begin{split}
			(\mathcal G_\omega h)(z)
			&=\frac12\sum_{i=1}^2
			\bigl[h(z+e_i)+h(z-e_i)-2h(z)\bigr]\\
			&\quad+\frac12\sum_{i=1}^2\omega(z,i)
			\bigl[h(z+e_i)-h(z-e_i)\bigr]\, .
		\end{split}
	\end{equation}
	On $\ell^2(\Z^2)$, the first sum is the self-adjoint simple-random-walk
	generator.  The second sum is skew-adjoint because $\omega(z,i)$ is
	constant along $e_i$.
	
	For $x\in\Z^2$, define the environment shift and its action on functions by
	\[
	(T_xF)(\omega)=F(\tau_x\omega),\qquad
	(\tau_x\omega)(z,i)=\omega(z+x,i)\, .
	\]
	The operator $T_x$ is unitary on $L^2(\P)$, and $T_{\pm e_i}$ commutes
	with multiplication by $\omega(0,i)$ because that sign is constant along
	the line parallel to $e_i$.
	
	We keep track of both the walker's position and the environment seen
	from that position.  The generator of this pair on
	$L^2(\P;\ell^2(\Z^2))$ is
	\begin{equation}\label{eq:joint-generator}
		(\mathcal GF)(\omega,z)=\sum_{i=1}^2
		\bigl[F(\tau_{\omega(0,i)e_i}\omega,z+\omega(0,i)e_i)-F(\omega,z)\bigr]\, .
	\end{equation}
	For returns to the origin, set
	$\delta_0(\omega,z)=\1_{\{z=0\}}$.  Then, the inner product being that of
	$L^2(\P;\ell^2(\Z^2))$,
	\begin{equation}\label{eq:return-correlation}
		\overline p_t(0,0)=\ip{\delta_0}{e^{t\mathcal G}\delta_0}\, .
	\end{equation}
	Under the Fourier transform
	$\widehat F(\omega,p)=\sum_z e^{-ip\cdot z}F(\omega,z)$, the function
	$F(\tau_x\omega,z+x)$ becomes $e^{ip\cdot x}T_x\widehat F$.
	
	\begin{proposition}\label{prop:generator}
		After Fourier transformation in the position variable, $\mathcal G$ acts
		at each frequency $p\in\T^2$ as $G_p=S_p+A_p$, where
		\begin{align}
			S_p&=\frac12\sum_{i=1}^2
			\bigl(e^{ip_i}T_{e_i}+e^{-ip_i}T_{-e_i}-2I\bigr),\label{eq:Sp}\\
			A_p&=\frac12\sum_{i=1}^2\omega(0,i)
			\bigl(e^{ip_i}T_{e_i}-e^{-ip_i}T_{-e_i}\bigr)\, .\label{eq:Ap}
		\end{align}
		Uniformly in $p$, the operator $S_p$ is self-adjoint and nonpositive,
		$A_p^*=-A_p$, $\norm{S_p}\leq4$, and $\norm{A_p}\leq2$.
		Let $d(s)\coloneqq1-\cos s$ and $\theta(p)\coloneqq d(p_1)+d(p_2)$.
		For every $\lambda>0$,
		\begin{align}
			\int_0^\infty e^{-\lambda t}\overline p_t(0,0)\dd t
			&=\int_{\T^2}r_\lambda(p)\dd m(p),\label{eq:green}\\
			r_\lambda(p)&\coloneqq
			\operatorname{Re}\ip1{(\lambda I-G_p)^{-1}1}\, .\label{eq:frequency-resolvent}
		\end{align}
	\end{proposition}
	
	\begin{proof}
		Splitting \eqref{eq:joint-generator} according to the two values of each
		sign gives \eqref{eq:Sp} and \eqref{eq:Ap}.  Unitarity of the shifts and
		their commutation with the corresponding signs give the operator
		assertions.  The Fourier transform of $\delta_0$ is $1$, so Plancherel's
		theorem and integration of \eqref{eq:return-correlation} against
		$e^{-\lambda t}\dd t$ give \eqref{eq:green}.
	\end{proof}
	
	\subsection{The fixed-frequency bound}\label{ssec:varbound}
	
	Fix $\lambda>0$ and $p\in\T^2$, and let
	\begin{equation}\label{eq:H}
		H\coloneqq\lambda I-S_p,\qquad
		\lambda I\leq H\leq(\lambda+4)I,\qquad
		H1=(\lambda+\theta(p))1\, .
	\end{equation}
	Thus $H^{-1}$ and $H^{-1/2}$ are bounded, and we set
	\begin{equation}\label{eq:Hnorms}
		\Hplus g^2\coloneqq\ip g{Hg},\qquad
		\Hminus u^2\coloneqq\ip u{H^{-1}u}\, .
	\end{equation}
	
	\begin{lemma}\label{lem:variational-bound}
		For every $\lambda>0$, $p\in\T^2$, and $g\in L^2(\P)$,
		\begin{equation}\label{eq:var}
			0\leq r_\lambda(p)\leq\Hplus g^2+\Hminus{1-A_pg}^2\, .
		\end{equation}
	\end{lemma}
	
	\begin{proof}
		Let $z\coloneqq(H-A_p)^{-1}1$, $d\coloneqq z+g$, and $q\coloneqq1-A_pg$.  Then
		$q=Hz-A_pd$, and skew-adjointness gives
		\[
		r_\lambda(p)=\operatorname{Re}\ip1z=\Hplus z^2,
		\qquad
		\operatorname{Re}\ip qd=\operatorname{Re}\ip{Hz}d\, .
		\]
		Expanding $z=d-g$ gives
		\[
		\Hplus z^2=\Hplus g^2
		+\left(2\operatorname{Re}\ip qd-\Hplus d^2\right)\, .
		\]
		Weighted Cauchy--Schwarz bounds the bracket by $\Hminus q^2$.
	\end{proof}
	
	\begin{remark}
		Taking $g=0$ gives
		\begin{equation}\label{eq:zero-function}
			r_\lambda(p)\leq\Hplus0^2+\Hminus1^2
			=\frac1{\lambda+\theta(p)}\, .
		\end{equation}
		This bound suffices when $p$ stays away from $0$.
	\end{remark}
	
	\section{Fourier--Walsh coordinates}\label{sec:coins}
	
	We evaluate the variational bound by expanding the test function in
	products of the orientations $\omega(x,j)$.  Fix $\lambda>0$ and $p\in\T^2$, and write
	$A=A_p$.
	
	For a finite set of lines, form the products over all subsets of their
	orientations, including the empty product $1$.  These products form an
	orthonormal basis for functions of those orientations
	\citep*[Theorem~1.5]{ODo14}.  Taking conditional expectations onto
	increasing finite sets that exhaust all lines and applying $L^2$
	martingale convergence gives the Fourier--Walsh basis of $L^2(\P)$.
	
	\subsection{Degree and frequency}\label{ssec:degree}
	
	Index the horizontal lines by $(ye_2,1)$ and the vertical lines by
	$(xe_1,2)$, with $x,y\in\Z$.  For $n\geq0$, let $\Hs_n$ be the closed
	span of products of the orientations of $n$ distinct lines.
	Every $F\in\Hs_n$ has a square-summable coefficient array $f$, symmetric
	under permutations of the line indices $(x_a,j_a)$, such that
	\[
	F(\omega)=\sum_{\mathbf x,\mathbf j}f_{\mathbf j}(\mathbf x)
	\prod_{a=1}^n\omega(x_a,j_a)\, .
	\]
	The sum converges in $L^2(\P)$ and runs over ordered line indices;
	$f$ vanishes when two line indices coincide.
	If $F\in\Hs_n$ and
	$G\in\Hs_m$ have coefficients $f$ and $g$, then
	\begin{equation}\label{eq:iso}
		\ip FG=\1_{\{m=n\}}n!\ip fg_{\ell^2}\, .
	\end{equation}
	The factor $n!$ counts the orderings of the $n$ line indices.
	
	Multiplication by $\omega(0,i)$ introduces this factor into a product
	when it is absent and cancels it when it is present, since
	$\omega(0,i)^2=1$.  Consequently
	\begin{equation}\label{eq:DA}
		A=D-D^*,\qquad D\colon\Hs_n\longrightarrow\Hs_{n+1}\, .
	\end{equation}
	Here $D$ raises the degree by one and $-D^*$ lowers it by one, as in the
	decomposition in \citet*[(2.6)--(2.8)]{LQSY04} for the exclusion process
	at density $1/2$.
	To exclude repeated line indices after raising, let $\Pi_n$ denote the
	orthogonal projection onto coefficients supported on distinct line indices.
	Write $H_n$ and $D_n$ for the restrictions of $H$ and $D$ to $\Hs_n$.
	
	For the line coefficients we use the Fourier convention
	\begin{equation}\label{eq:line-fourier}
		\widehat f_{\mathbf j}(\boldsymbol\xi)
		=\sum_{\mathbf x}e^{i\sum_a x_a\cdot\xi_a}f_{\mathbf j}(\mathbf x)\, ,
	\end{equation}
	with convergence in $L^2$.  A line frequency $\xi_a$ has zero $j_a$th
	coordinate.  Simultaneous shifts now act by scalar multiplication.
	We drop the hats and also use $\Pi_n$ for the projection after Fourier
	transformation.
	For a degree-$n$ coefficient, the total frequency is
	\begin{equation}\label{eq:P}
		P=p+\xi_1+\cdots+\xi_n\, .
	\end{equation}
	Equations~\eqref{eq:Sp} and \eqref{eq:H} give
	\begin{equation}\label{eq:Hsym}
		\text{$H_n$ is multiplication by }\lambda+\theta(P)\, .
	\end{equation}
	We will use
	\begin{equation}\label{eq:sine}
		\log_+x=\max\{0,\log x\},
		\qquad \frac{2s^2}{\pi^2}\leq d(s)\leq\frac{s^2}{2}
		\quad (\abs s\leq\pi)\, .
	\end{equation}
	We identify $H_0$ with the scalar $\lambda+\theta(p)$ by which it acts
	on constants.
	
	\begin{lemma}\label{lem:raise}
		Let $\widetilde D_n$ be the raising operator computed without imposing
		distinct line indices.  Then
		\begin{equation}\label{eq:raise}
			(\widetilde D_nf)_{j_1\dots j_{n+1}}
			=\frac i{n+1}\sum_{a=1}^{n+1}\sin(P_{j_a})
			f_{j_1\dots\widehat{j_a}\dots j_{n+1}}\, ,
		\end{equation}
		where $P$ is the total frequency after raising and the hat omits the
		$a$-th line index.  Moreover,
		$D_n=\Pi_{n+1}\widetilde D_n\Pi_n$, and $\Pi_n$ commutes with every bounded
		multiplier that is a function of $P$.
	\end{lemma}
	
	The commutation statement lets us remove a projection in upper bounds.
	If $\Pi$ is an orthogonal projection commuting with a bounded positive
	multiplier $N$ with bounded inverse, then
	\begin{equation}\label{eq:contract}
		\ip{\Pi u}{N^{-1}\Pi u}\leq\ip u{N^{-1}u},\qquad
		\ip{\Pi u}{N\Pi u}\leq\ip u{Nu}\, .
	\end{equation}
	
	\begin{proof}[Proof of Lemma~\ref{lem:raise}]
		By \eqref{eq:Ap}, raising first shifts and then multiplies by
		$\omega(0,i)$.  On a coefficient of total frequency $P'$, the shift
		contributes
		$\tfrac12(e^{iP'_i}-e^{-iP'_i})=i\sin(P'_i)$.  Multiplication appends a line
		index of type $i$ at the site $0$, whose Fourier factor is $1$.
		Symmetrization averages over the $n+1$ positions of the new index.  If it
		occupies position $a$, then $P'=P-\xi_a$ and $(\xi_a)_{j_a}=0$, so
		$P'_{j_a}=P_{j_a}$.  This gives \eqref{eq:raise}.  The new index must be
		distinct from the existing indices, which gives the projection $\Pi_{n+1}$.
		
		Simultaneous translations preserve distinctness, so $\Pi_n$ commutes
		with their multipliers.  Uniform approximation by trigonometric polynomials
		gives commutation with continuous functions of $P$ in operator norm;
		spectral calculus extends this to bounded Borel functions of $P$.
	\end{proof}
	
	\subsection{Degrees one and three}\label{ssec:degrees-one-three}\label{ssec:low-coordinates}
	
	We now take $g=f+k$ with $f\in\Hs_1$ and $k\in\Hs_3$.
	After separating the contribution of the constant component of
	$1-A_p(f+k)$, the remaining terms in \eqref{eq:var} are
	\begin{equation}\label{eq:E}
		E_p(f,k)=\ip f{H_1f}+\ip k{H_3k}
		+\Hminus{D_1f-D_2^*k}^2+\Hminus{D_3k}^2\, .
	\end{equation}
	Take $f$ and $k$ of the form \eqref{eq:overview-test-family}.
	On coefficients of products $\omega(ye_2,1)\omega(xe_1,2)$, Fourier
	transformation turns $H$ into multiplication by
	$\lambda+d(p_2+s)+d(p_1+u)$, where $s$ and $u$ are Fourier dual to $y$
	and $x$.  Let $s'$ be the frequency of the second horizontal line index
	in $k$.  We center the denominator by setting
	\begin{equation}\label{eq:shift}
		r=p_2+s,\qquad r'=p_2+s',\qquad \beta=p_1+u,
		\qquad\text{and}\qquad \alpha=r+r'-p_2\, .
	\end{equation}
	Here $r,r'$ are the shifted frequencies of the two horizontal line
	indices in $k$, and $\beta$ is the shifted frequency of its vertical
	line index.  Subscripts list the values of $j$ in the factors
	$\omega(x,j)$; for example, $112$ denotes two horizontal orientations
	and one vertical orientation.  The total frequencies are
	\[
	\renewcommand{\arraystretch}{1.15}
	\begin{array}{llll}
		\text{degree $0$:} & P=(p_1,p_2), &\quad
		\text{component $11$:} & P=(p_1,\alpha),\\
		\text{component $1$:} & P=(p_1,r), &\quad
		\text{component $12$:} & P=(\beta,r),\\
		& &\quad \text{component $112$:} & P=(\beta,\alpha)\, .
	\end{array}
	\]
	The shifts preserve Haar measure.  We use these variables throughout the
	coefficient calculations.

	\subsection{Low-degree operators}\label{ssec:low-operators}
	
	By \eqref{eq:iso}, multiplying the Fourier coefficients of components
	$11$, $12$, and $112$ by $\sqrt2$, $2$, and $\sqrt{18}$, respectively, makes
	the correspondence with Walsh expansions an isometry.  From now on,
	we identify each Walsh component with its normalized Fourier coefficient
	function.
	
	\begin{lemma}\label{lem:formulas}
		Let $k(r,r',\beta)$ be a normalized Fourier coefficient of component
		$112$, symmetric in $r,r'$.  In the variables \eqref{eq:shift} and the
		normalization above,
		\begin{align}
			\bigl(\widetilde D_2^*k\bigr)_{11}(r,r')
			&=-i\sin(\alpha)\int_\T k(r,r',\beta)\dd m(\beta),\label{eq:D2a}\\
			\bigl(\widetilde D_2^*k\bigr)_{12}(r,\beta)
			&=-i\sqrt2\sin(\beta)\int_\T k(r,r',\beta)\dd m(r')\, .\label{eq:D2b}
		\end{align}
		The first formula integrates out the frequency of the vertical line
		index; the second integrates out that of either horizontal line index.
		For $f$ linear in the orientations $\omega(ye_2,1)$,
		\begin{equation}\label{eq:D1}
			(D_0^*f)(p)=-i\sin (p_1)\int_\T f\dd m,\quad
			(\widetilde D_1f)_{11}=\tfrac i{\sqrt2}\sin (p_1)[f(r)+f(r')],\quad
			(\widetilde D_1f)_{12}=i\sin (r)f(r)\, .
		\end{equation}
	\end{lemma}
	
	\begin{proof}[Proof of Lemma~\ref{lem:formulas}]
		For a coefficient with subscript $112$, equation~\eqref{eq:raise} gives, before
		normalization,
		\[
		\frac i3\bigl[\sin\beta f_{12}(r',\beta)
		+\sin\beta f_{12}(r,\beta)+\sin\alpha f_{11}(r,r')\bigr]\, .
		\]
		Multiplying the coefficients with subscripts $11$, $12$, and $112$ by
		$\sqrt2$, $2$, and $\sqrt{18}=3\sqrt2$, then taking adjoints, gives
		\eqref{eq:D2a}--\eqref{eq:D2b}; symmetry in $(r,r')$ gives the factor
		$\sqrt2$ in \eqref{eq:D2b}.
		
		For the Fourier coefficient of $f$, the total frequency is $(p_1,r)$.
		Equation~\eqref{eq:raise} with $n=0$ gives
		$(\widetilde D_0u)_1=i\sin(p_1)u$, whose adjoint is the first formula in
		\eqref{eq:D1}.  Equation~\eqref{eq:raise} with $n=1$, followed by the
		degree-two normalization above, gives the other two formulas.
	\end{proof}
	
	\section{The test function}\label{sec:evaluation}
	
	Exchange rows and columns if necessary, and fix
	$\lambda\in(0,1]$ and $p\in\T^2$ with
	$a\coloneqq\abs{p_1}\geq\abs{p_2}$.
	Let
	\begin{equation}\label{eq:degree-one-parameters}
		\delta\coloneqq\sqrt\lambda+a,\qquad
		s\coloneqq\sin p_1,\qquad H_0\coloneqq\lambda+\theta(p)\, .
	\end{equation}
	We use the shifted variables of \eqref{eq:shift}; sums are taken modulo
	$2\pi$, and absolute values use representatives in $(-\pi,\pi]$.
	Let $\varphi\in L^2(\T,m)$ be real and even, and set
	\begin{equation}\label{eq:f}
		f(r)=-i\varphi(r),\qquad w(r)\coloneqq\sin(r)\varphi(r)\, .
	\end{equation}
	Equation~\eqref{eq:D1} gives
	\begin{equation}\label{eq:degree-one-images}
		D_0^*f=-s\int_\T\varphi\dd m,\qquad (D_1f)_{12}=w\, .
	\end{equation}
	We choose $k$ so that $(D_2^*k)_{12}$ approximates $w$ and
	$(D_2^*k)_{11}=0$.  The correction will add a logarithmic
	term to the denominator $\lambda+d(r)+d(\beta)$, giving
	\begin{equation}\label{eq:mixed-weight}
		\begin{split}
			W(r,\beta)\coloneqq{}&\lambda+d(r)+d(\beta)\\
			&+\frac{\sin^2\beta}{4\pi}
			\log\left(1+\frac\pi{\delta+\abs r+\abs\beta}\right)\, .
		\end{split}
	\end{equation}
	The following lemma bounds the sum of the energy of $k$, the quartic
	contribution $\Hminus{D_3k}^2$, and the remaining quadratic contribution.
	
	\begin{lemma}\label{lem:parity}
		For every $\lambda\in(0,1]$, $p\in\T^2$ with
		$\abs{p_2}\leq\abs{p_1}$, and real even
		$\varphi\in L^2(\T,m)$, there is $k\in\Hs_3$ of the form
		\eqref{eq:overview-test-family} such that
		$(D_2^*k)_{11}=0$ and
		\begin{equation}\label{eq:correction-cost}
			\begin{split}
				\Hplus k^2+\Hminus{D_3k}^2
				+\Hminus{(D_1f)_{12}-(D_2^*k)_{12}}^2
				\leq\int_{\T^2}\frac{w(r)^2}{W(r,\beta)}
				\dd m(r)\dd m(\beta)\, .
			\end{split}
		\end{equation}
	\end{lemma}
	
	\begin{proof}
		\noindent\emph{Step 1.}
		To preserve parity in the minimization, we bound
		$\Hplus k^2+\Hminus{D_3k}^2$ by a weighted squared norm of its Fourier
		coefficient, with a weight even in each frequency.
		Set
		\begin{equation}\label{eq:M}
			M(r,r',\beta)\coloneqq
			4\bigl(\delta+\abs r+\abs{r'}+\abs\beta\bigr)\, .
		\end{equation}
		For every $K\in L^2(\T^3)$ symmetric in $r,r'$, we have
		\begin{equation}\label{eq:majorant}
			\Hplus{\Pi_3K}^2+\Hminus{D_3\Pi_3K}^2
			\leq\int_{\T^3}M\abs K^2
			\dd m(r)\dd m(r')\dd m(\beta)\, .
		\end{equation}
		To integrate over the frequency of a newly added line index, use
		\begin{equation}\label{eq:line-integral}
			\int_\T\frac{\dd m(t)}{\mu+d(t)}
			=\frac1{\sqrt{\mu(\mu+2)}}\leq\frac1{\sqrt{2\mu}}
			\qquad(\mu>0)\, .
		\end{equation}
		The substitution $u=\tan(t/2)$ gives the equality.
		By \eqref{eq:iso}, the normalization factors for subscripts $1112$, $1122$,
		and $112$ are $\sqrt{96}$, $12$, and $\sqrt{18}$.
		Equation~\eqref{eq:raise} therefore gives $i\sin\beta/\sqrt3$ times
		three summands for component $1112$, and $i\sin\alpha/\sqrt2$ times
		two summands for component $1122$.
		Drop the degree-four projection $\Pi_4$ using \eqref{eq:contract}, bound each squared
		sum by its number of terms times the sum of squares, and apply
		\eqref{eq:line-integral}.  The result is
		\begin{align*}
			\Hplus{\Pi_3K}^2+\Hminus{D_3\Pi_3K}^2
			\leq\int_{\T^3}\Biggl[&\lambda+d(\beta)+d(\alpha)
			+\frac{3\sin^2\beta}
			{\sqrt{(\lambda+d(\beta))(\lambda+d(\beta)+2)}}\\
			&+\frac{2\sin^2\alpha}
			{\sqrt{(\lambda+d(\alpha))(\lambda+d(\alpha)+2)}}\Biggr]
			\abs{\Pi_3K}^2\dd m(r)\dd m(r')\dd m(\beta)\, .
		\end{align*}
		The bracket is a positive continuous function of the total frequency
		$(\beta,\alpha)$, so Lemma~\ref{lem:raise} and \eqref{eq:contract}
		allow $\Pi_3K$ to be replaced by $K$ in this integral.
		The inequalities $d(t)\leq\abs t$ and
		$\sin^2t\leq\abs t\sqrt{2d(t)}$ bound the bracket by
		$\lambda+4\abs\beta+3\abs\alpha\leq M$, since
		$\lambda\leq\sqrt\lambda$ and $\abs\alpha\leq\abs r+\abs{r'}+a$.
		This proves \eqref{eq:majorant}; removing the projection before comparing
		with $M$ avoids requiring $\Pi_3$ to commute with $M$.
		
		\medskip
		\noindent\emph{Step 2.}
		Write $W=B+\sigma$, where
		\begin{align}
			B(r,\beta)&\coloneqq\lambda+d(r)+d(\beta),\label{eq:B}\\
			\sigma(r,\beta)&\coloneqq W(r,\beta)-B(r,\beta)
			=\sin^2\beta\int_\T\frac{\dd m(r')}{M(r,r',\beta)}\, .\label{eq:sigma}
		\end{align}
		Integrating $1/(\delta+\abs r+\abs\beta+\abs{r'})$ over $r'$ gives
		the logarithm in \eqref{eq:mixed-weight}.
		For every real square-integrable $t(r,\beta)$ odd in $r$ and even in
		$\beta$, the construction below gives a symmetric coefficient $K$ with
		$(\widetilde D_2^*K)_{12}=\sigma t$ and
		$\int_{\T^3}M\abs K^2\dd m=\int_{\T^2}\sigma t^2\dd m$.
		Thus the minimization reduces to completing the square at each
		$(r,\beta)$:
		\begin{equation}\label{eq:scalar-energy}
			\sigma t^2+\frac{(w-\sigma t)^2}{B}
			=\frac{w^2}{W}+\frac{\sigma W}{B}
			\left(t-\frac wW\right)^2
			\qquad(t\in\R)\, .
		\end{equation}
		Taking $t=w/W$, define the symmetric coefficient
		\begin{equation}\label{eq:K}
			K(r,r',\beta)\coloneqq
			\frac{i\sin\beta}{\sqrt2 M(r,r',\beta)}
			\left[\frac{w(r)}{W(r,\beta)}+
			\frac{w(r')}{W(r',\beta)}\right]\, .
		\end{equation}
		Since $W\geq\lambda$ and $M\geq4\delta$, we have $K\in L^2(\T^3)$.
		The function $w(r)/W(r,\beta)$ is odd in $r$ and even in $\beta$, while
		$M$ is even in each variable.  Hence
		\begin{equation}\label{eq:mixed-cancellation}
			\int_\T\frac{w(r')}{M(r,r',\beta)W(r',\beta)}\dd m(r')=0,
			\qquad (\widetilde D_2^*K)_{12}=\frac{\sigma w}{W}\, .
		\end{equation}
		
		In the expansion of $M\abs K^2$, the cross term integrates to zero by
		\eqref{eq:mixed-cancellation}.  The two square terms have equal integrals
		by symmetry, so
		\begin{equation}\label{eq:parity-energy}
			\int_{\T^3}M\abs K^2\dd m(r)\dd m(r')\dd m(\beta)
			=\int_{\T^2}\frac{\sigma w^2}{W^2}\dd m(r)\dd m(\beta)\, .
		\end{equation}
		The same calculations apply with any real square-integrable $t(r,\beta)$
		odd in $r$ and even in $\beta$ in place of $w/W$.
		These identities give the quadratic expression in \eqref{eq:scalar-energy}.
		
		\medskip
		\noindent\emph{Step 3.}
		Set $k=\Pi_3K\in\Hs_3$ to remove coefficients with coincident row indices.
		The removed part has Fourier transform
		\[
		((I-\Pi_3)K)(r,r',\beta)
		=\int_\T K(t,r+r'-t,\beta)\dd m(t)\, .
		\]
		Its integral in $r'$ equals the integral of $K$ over both row variables.
		This double integral vanishes because $K$ is odd under
		$(r,r')\mapsto(-r,-r')$.  Thus
		\begin{equation}\label{eq:projected-lowering}
			\int_\T((I-\Pi_3)K)(r,r',\beta)\dd m(r')=0,
			\qquad (D_2^*k)_{12}=\frac{\sigma w}{W}\, .
		\end{equation}
		Also, $K$ is odd in $\beta$, and $\Pi_3$ acts only on row indices.
		Equation~\eqref{eq:D2a} gives $(D_2^*k)_{11}=0$.
		
		Finally, $w-\sigma w/W=Bw/W$.
		Equations~\eqref{eq:majorant}, \eqref{eq:parity-energy}, and
		\eqref{eq:projected-lowering} therefore bound the left side of
		\eqref{eq:correction-cost} by
		\[
		\int_{\T^2}\frac{(\sigma+B)w^2}{W^2}\dd m(r)\dd m(\beta)
		=\int_{\T^2}\frac{w^2}{W}\dd m(r)\dd m(\beta)\, .
		\]
	\end{proof}
	
	\section{Estimating the quadratic functional}\label{sec:effective-energy}
	
	We integrate out the column frequency to reduce the variational bound
	to a quadratic functional of $\varphi$ alone.
	
	Restrict the row frequencies to
	\begin{equation}\label{eq:Gamma-q}
		\Gamma_\delta\coloneqq
		\{r\in\T:\sqrt\delta\leq\abs r\leq1/4\}
		\qquad(0<\delta<1/256)\, .
	\end{equation}
	The cutoff gives $\delta\leq r^2$, so the degree-one contribution can
	be bounded using the weight
	\begin{equation}\label{eq:q}
		q(r)\coloneqq\frac{\abs r}{\sqrt{\log(1/\abs r)}}
		\qquad(0<\abs r\leq1/4)\, .
	\end{equation}
	
	\begin{lemma}\label{lem:effective-energy}
		For every
		$\lambda\in(0,1]$, $p\in\T^2$ with
		$\abs{p_2}\leq\abs{p_1}$ and
		$\delta=\sqrt\lambda+\abs{p_1}<1/256$, and every real even
		$\varphi\in L^2(\T,m)$ vanishing off $\Gamma_\delta$,
		\begin{equation}\label{eq:effective-energy}
			r_\lambda(p)\leq
			\frac{\left(1-s\int_{\Gamma_\delta}\varphi\dd m\right)^2}{H_0}
			+16\int_{\Gamma_\delta}q(r)\varphi(r)^2\dd m(r)\, .
		\end{equation}
	\end{lemma}
	
	For each fixed $\delta$, the weight $q$ is bounded above and away from
	zero on $\Gamma_\delta$.  Thus the lemma covers every real even measurable
	$\varphi$ supported there with finite weighted square integral.
	
	\begin{proof}
		Take $f$ and $k$ as in Lemma~\ref{lem:parity}.
		By \eqref{eq:DA}, \eqref{eq:E}, and \eqref{eq:degree-one-images},
		orthogonality of the degrees gives
		\begin{equation}\label{eq:uncancelled}
			\Hplus{f+k}^2+\Hminus{1-A_p(f+k)}^2
			=\frac{\left(1-s\int_\T\varphi\dd m\right)^2}{H_0}+E_p(f,k)\, .
		\end{equation}
		Set $\mu=\lambda+d(p_1)$.  Apply
		$\abs{f(r)+f(r')}^2\leq2\abs{f(r)}^2+2\abs{f(r')}^2$ in
		\eqref{eq:D1}, remove the projection using \eqref{eq:contract}, and
		integrate with \eqref{eq:line-integral}.  Together with \eqref{eq:Hsym},
		this gives
		\begin{equation}\label{eq:degree-one-cost}
			\begin{split}
				\Hplus f^2+\Hminus{(D_1f)_{11}}^2
				&\leq\int_\T\left[\mu+d(r)+
				\frac{2s^2}{\sqrt{\mu(\mu+2)}}\right]\varphi(r)^2\dd m(r)\\
				&\leq\int_\T\left(\delta^2+2\sqrt2\delta+\frac{r^2}{2}\right)
				\varphi(r)^2\dd m(r)\, .
			\end{split}
		\end{equation}
		Here we used $\mu\leq\delta^2$ and $s^2\leq2\mu$.
		On $\Gamma_\delta$, the last weight is at most $5r^2\leq5q(r)$,
		since $\delta\leq r^2$.  Lemma~\ref{lem:parity} therefore gives
		\begin{equation}\label{eq:energy-reduction}
			E_p(f,k)\leq5\int_{\Gamma_\delta}q(r)\varphi(r)^2\dd m(r)
			+\int_{\T^2}\frac{\sin^2r\varphi(r)^2}{W(r,\beta)}
			\dd m(r)\dd m(\beta)\, .
		\end{equation}
		It remains to bound the integral of $W^{-1}$ in $\beta$.
		Fix $r\in\Gamma_\delta$, let $\rho\coloneqq\abs r$, and split this integral at
		$\abs\beta=\sqrt\rho$.
		On the inner interval, $\delta\leq\rho^2$ gives
		$\delta+\rho+\abs\beta\leq3\sqrt\rho$ and
		$\abs\beta\leq1/2$.  Therefore
		\[
		\log\left(1+\frac\pi{\delta+\rho+\abs\beta}\right)
		\geq\frac12\log(1/\rho),\qquad
		W(r,\beta)\geq\frac{2\rho^2}{\pi^2}
		+\frac{\beta^2\log(1/\rho)}{2\pi^3}\, .
		\]
		The second inequality follows from \eqref{eq:mixed-weight},
		\eqref{eq:sine}, and $\abs{\sin\beta}\geq2\abs\beta/\pi$
		for $\abs\beta\leq\pi/2$.
		On the outer interval, $W\geq d(\beta)\geq2\beta^2/\pi^2$.
		Hence
		\begin{equation}\label{eq:beta}
			\begin{split}
				\int_\T\frac{\dd m(\beta)}{W(r,\beta)}
				&\leq\frac1{2\pi}\int_\R
				\frac{\dd\beta}{2\rho^2/\pi^2+
					\beta^2\log(1/\rho)/(2\pi^3)}
				+\frac\pi2\int_{\sqrt\rho}^\pi\frac{\dd\beta}{\beta^2}\\
				&\leq\frac{\pi^{5/2}}{2\rho\sqrt{\log(1/\rho)}}
				+\frac\pi{2\sqrt\rho}
				\leq\frac{11}{\rho\sqrt{\log(1/\rho)}}\, .
			\end{split}
		\end{equation}
		The last inequality uses $\rho\log(1/\rho)\leq1$ and
		$(\pi^{5/2}+\pi)/2<11$.
		Since $\sin^2r\leq r^2$, equations~\eqref{eq:energy-reduction} and
		\eqref{eq:beta} give $E_p(f,k)\leq16\int_{\Gamma_\delta}q\varphi^2\dd m$.
		Substitution in \eqref{eq:uncancelled} and Lemma~\ref{lem:variational-bound}
		prove the result.
	\end{proof}
	
	\section{The frequency bound}\label{sec:fixed-frequency}
	
	The first term on the right side of \eqref{eq:effective-energy} depends only on
	$\int_{\Gamma_\delta}\varphi\dd m$.
	For a fixed value of this integral, Cauchy--Schwarz shows that
	$\int_{\Gamma_\delta}q\varphi^2\dd m$ is minimized when $\varphi$ is proportional to
	$q^{-1}$ on $\Gamma_\delta$.  This leaves one scalar to optimize.  Set
	\begin{equation}\label{eq:Z}
		\begin{split}
			Z_\delta\coloneqq\int_{\Gamma_\delta}\frac{\dd m(r)}{q(r)}
			&=\frac{2}{3\pi}\left[
			\left(\tfrac12\log(1/\delta)\right)^{3/2}
			-(\log4)^{3/2}\right]\\
			&\asymp[\log(1/\delta)]^{3/2}
			\qquad(0<\delta<1/256)\, .
		\end{split}
	\end{equation}
	The equality follows by substituting $u=\log(1/r)$.
	Choose $\varphi(r)=t/q(r)$ on $\Gamma_\delta$ and zero elsewhere,
	where $t\in\R$.  This function belongs to $L^2(\T,m)$ because
	$\Gamma_\delta$ stays away from zero.  Then
	$\int\varphi\dd m=tZ_\delta$ and
	$\int q\varphi^2\dd m=t^2Z_\delta$.
	Substituting in \eqref{eq:effective-energy} and completing the square gives
	\begin{equation}\label{eq:scalar-minimum}
		\inf_{t\in\R}\left\{\frac{(1-stZ_\delta)^2}{H_0}
		+16t^2Z_\delta\right\}
		=\frac1{H_0+s^2Z_\delta/16},\qquad
		t=\frac{s}{16H_0+s^2Z_\delta}\, .
	\end{equation}
	Thus
	\begin{equation}\label{eq:optimized}
		r_\lambda(p)\leq\frac1{H_0+s^2Z_\delta/16}\, .
	\end{equation}
	
	\begin{proposition}\label{prop:frequency}
		For every
		$\lambda\in(0,1]$ and $p\in\T^2$, writing
		$a(p)=\max\{\abs{p_1},\abs{p_2}\}$, we have
		\begin{equation}\label{eq:frequency-bound}
			r_\lambda(p)\leq
			\frac{2048}{\lambda+a(p)^2
				[1+\log_+(1/(\sqrt\lambda+a(p)))]^{3/2}}\, .
		\end{equation}
	\end{proposition}
	
	\begin{proof}
		Set $\delta=\sqrt\lambda+a(p)$.
		For $\delta<1/256$, exchange rows and columns if needed.
		Equation~\eqref{eq:Z} gives $Z_\delta\geq[\log(1/\delta)]^{3/2}/32$.
		Since $s^2\geq4a(p)^2/\pi^2$ and $\log(1/\delta)>5$,
		\[
		H_0+\frac{s^2Z_\delta}{16}
		\geq\lambda+\frac{a(p)^2[\log(1/\delta)]^{3/2}}{128\pi^2}
		\geq\frac{\lambda+a(p)^2[1+\log(1/\delta)]^{3/2}}{2048}\, .
		\]
		Thus \eqref{eq:optimized} gives \eqref{eq:frequency-bound}.
		For $\delta\geq1/256$, use \eqref{eq:zero-function},
		$\theta(p)\geq a(p)^2/5$, and
		$[1+\log_+(1/\delta)]^{3/2}<19$.
	\end{proof}
	
	\begin{proof}[Proof of Theorem~\ref{thm:annealed}]
		The square $\{a(p)\leq b\}$ has normalized measure $b^2/\pi^2$ for
		$0\leq b\leq\pi$.  On $a(p)\leq\sqrt\lambda$, the bound
		$r_\lambda\leq1/\lambda$ therefore contributes at most $1/\pi^2$.
		On $\sqrt\lambda<a(p)<1/4$,
		$1+\log_+(1/(\sqrt\lambda+a(p)))\geq[1+\log(1/a(p))]/2$.
		On $a(p)\geq1/4$, use $r_\lambda(p)\leq\pi^2/(2a(p)^2)$.
		Integrating these three bounds using Proposition~\ref{prop:frequency} gives
		\begin{align*}
			\int_{\T^2}r_\lambda(p)\dd m(p)
			&\leq\frac1{\pi^2}
			+\frac{2^{5/2}\cdot2048}{\pi^2}
			\int_0^{1/4}\frac{\dd a}{a[1+\log(1/a)]^{3/2}}+\log(4\pi)\\
			&=\frac1{\pi^2}
			+\frac{2^{7/2}\cdot2048}{\pi^2\sqrt{1+\log4}}+\log(4\pi)
			<2048\, .
		\end{align*}
		Here we integrated using $u=1+\log(1/a)$.
		For $\sqrt\lambda>1/4$, omit the middle integral; the other two bounds
		still cover all frequencies.  Equation~\eqref{eq:green} and monotone convergence
		as $\lambda\downarrow0$ prove the explicit Green bound
		\begin{equation}\label{eq:green-explicit}
			\int_0^\infty\overline p_t(0,0)\dd t\leq2048\, .
		\end{equation}
	\end{proof}
	
	\begin{proof}[Proof of Theorem~\ref{thm:main}]
		Tonelli's theorem gives
		$\int_0^\infty p_t^\omega(0,0)\dd t<\infty$ almost surely.
		The continuous-time walk is the discrete-time walk run with an independent
		rate-two Poisson clock, so another application of Tonelli gives
		\[
		\int_0^\infty p_t^\omega(0,0)\dd t
		=\sum_{n\geq0}p_n^\omega(0,0)
		\int_0^\infty e^{-2t}\frac{(2t)^n}{n!}\dd t
		=\frac12\sum_{n\geq0}p_n^\omega(0,0)\, .
		\]
		Finally, $p_n^\omega(x,x)=p_n^{\tau_x\omega}(0,0)$ and $\P$ is
		translation invariant.  Intersecting the probability-one events over
		$x\in\Z^2$ proves the theorem.
	\end{proof}
	
	\subsection{Return probabilities in time}\label{ssec:time}
	
	\begin{proposition}\label{prop:time}
		For every deterministic Manhattan orientation $\omega$, all $x,y\in\Z^2$,
		and $t\geq0$,
		\begin{equation}\label{eq:heat-kernel}
			p_t^\omega(x,y)\leq\frac2{t+2}\, .
		\end{equation}
		For $f(t)\coloneqq\overline p_t(0,0)$ and every $t>0$,
		\begin{equation}\label{eq:time-derivative}
			\abs{f'(t)}\leq\frac2{\sqrt t(1+t/4)}\leq8t^{-3/2}\, .
		\end{equation}
	\end{proposition}
	
	\begin{proof}
		We prove \eqref{eq:heat-kernel} by a Nash argument; see also
		\citet*[Proposition~3]{KT17}.  Fix $\omega$, write
		$P_t=e^{t\mathcal G_\omega}$, and let $S_\omega$ be the symmetric part
		of $\mathcal G_\omega$.  By \eqref{eq:discrete-drift}, $S_\omega$ is the
		rate-two simple-random-walk generator.  Fourier splitting gives
		\[
		\norm h_2^4\leq2\norm h_1^2\ip h{-S_\omega h}
		\qquad(h\in\ell^1(\Z^2))\, .
		\]
		Counting measure is invariant, so $P_t\delta_y$ and $P_t^*\delta_x$
		are nonnegative with $\ell^1$ norm one.  The energy identity and the
		Nash inequality bound both squared $\ell^2$ norms by $(1+t)^{-1}$.
		The semigroup identity and Cauchy--Schwarz then give
		\[
		p_t^\omega(x,y)
		=\ip{P_{t/2}^*\delta_x}{P_{t/2}\delta_y}
		\leq\frac1{1+t/2}\, .
		\]

		For $s>0$, differentiating the rate-two Poisson representation of $P_s$
		gives
		\[
		\sup_x\sum_y\abs{(\mathcal G_\omega P_s)(x,y)}
		\leq\E\abs{N_s/s-2}\leq\sqrt{2/s}\, ,
		\]
		where $N_s$ has the Poisson distribution of mean $2s$.
		Apply this bound and \eqref{eq:heat-kernel} to the two factors in
		$\mathcal G_\omega P_t=(\mathcal G_\omega P_{t/2})P_{t/2}$ to obtain
		$\abs{\partial_t p_t^\omega(x,y)}\leq2/[\sqrt t(1+t/4)]$.
		Averaging over $\omega$ proves \eqref{eq:time-derivative}.
	\end{proof}
	
	\section{Open questions}\label{sec:open}
	\begin{problem}[Mean-square displacement]\label{prob:powerlaw}
		Is the annealed mean-square displacement of order $t^{4/3}$ as
		$t\to\infty$?
		Can the renormalization scheme of \citet*[Section~3]{ABK26} be
		adapted to this setting?
	\end{problem}
	
	\begin{problem}[A decay rate]\label{prob:pointwise}
	Can one prove that $\overline p_t(0,0)=O(t^{-\beta})$ as $t\to\infty$
	for some $\beta>1$? Redner's prediction \citep*[(3)]{Red89} suggests return probabilities of order $t^{-4/3}$.
\end{problem}

	\begin{problem}[Which orientations are transient]\label{prob:which}
		Which orientation fields give a recurrent walk, and which give a transient
		one?  Theorem~\ref{thm:main} gives transience for $\P$-almost every
		environment.  A uniform random shift by a vector in $\{0,1\}^2$ of the
		alternating environment of Section~\ref{ssec:main-results} gives a stationary
		ergodic law with mean-zero line orientations under which the walk is
		recurrent almost surely.

		For the partially oriented walk, \citet*[Section~4]{GLN08T} ask a related
		question when an irrational rotation determines the horizontal orientations.
		Fix an irrational $\alpha$.  Direct the line at height $y\in\Z$ to the right
		if the fractional part of $y\alpha$ is less than $1/2$, and to the left
		otherwise; vertical edges can be crossed in either direction.
		Is this walk recurrent or transient?
	\end{problem}
	
	\bibliographystyle{plainnat}
	\bibliography{references}
	
\end{document}